\documentclass[preprint,11pt,xcolor=dvipsnames, a4wide]{amsart}

\usepackage{amsthm,amsmath,natbib, mathtools}
\usepackage{ifthen}
\usepackage{enumitem}
\usepackage{amssymb,dsfont,url,color,booktabs}
\usepackage[x11names]{xcolor}
 \usepackage{tikz}
 
 \usepackage{enumitem}
 \usepackage{geometry}
\usepackage{rotating}
\usepackage{epstopdf}
\usepackage{chngcntr}
\usepackage{bbm}
\usepackage{array}
\usepackage[export]{adjustbox}[2011/08/13] 
\usepackage{caption}
\usepackage{multirow}
\usepackage{todonotes}

\usepackage[hang]{subfigure} 
\usepackage{wrapfig} 
\usepackage{tikz}
\usetikzlibrary{arrows, patterns}
\tikzset{
  schraffiert/.style={pattern=horizontal lines,pattern color=#1},
  schraffiert/.default=black
}
\usepackage{todonotes}

\newcommand*{\IR}{\mathbb{R}}

\newcommand*{\IC}{\mathbb{C}}

\theoremstyle{plain}
\newtheorem{theorem}{Theorem}[section]

\newtheorem{corollary}[theorem]{Corollary}
\newtheorem{defin}[theorem]{Definition}
\newtheorem{remark}[theorem]{Remark}

\newtheorem{assumptions}[theorem]{Assumptions}

\newcommand{\dd}[1]{\operatorname{d}\!#1}
\newcommand{\ccd}{\mathds{C}^d}
\newcommand{\ee}[1]{\operatorname{e}^{#1}}

\newcommand{\OB}{\mathcal{B}}

\newcommand{\OI}{\mathcal{I}}
\newcommand{\II}{\mathcal{I}}

\newcommand{\OP}{\mathcal{P}}
\newcommand{\OQ}{\mathcal{Q}}

\newcommand{\OU}{\mathcal{U}}

\newcommand{\OX}{\mathcal{X}}

\newcommand{\notiz}[1]{\relax}

\newcommand{\zitep}[1]{\relax}
\newcommand{\skr}{\rangle}
\newcommand{\1}{\mathds 1}            

\newcommand{\nn}{\mathds N}

\newcommand{\rr}{\mathds R}

\newcommand{\rrd}{\mathds{R}^d}

\newcommand{\cc}{\mathds C}
\newcommand{\skl}{\langle}

\newcommand{\argmax}{\operatorname{arg\,max}}

\newcommand{\Price}[1][]{
		\ifthenelse{\equal{#1}{}}{\mathit{Price}}{\Price{}^{#1}}
	} 

\usepackage{mathrsfs}

\newcommand{\ulb}{\underline{b}}	
\newcommand{\olb}{\overline{b}}

\newcommand{\tild}{~}

\newlength{\wordlength}

\begin{document}




\title[Parametric Integration with Magic Points]{ Parametric Integration by Magic Point Empirical Interpolation}

\author{Maximilian Ga{\ss}}
\author{Kathrin Glau}

\date{\today\\\indent Technische Universit{\"a}t M{\"u}nchen, Center for Mathematics\\ \indent kathrin.glau@tum.de, maximilian.gass@mytum.de}
\maketitle



\begin{abstract}
We derive analyticity criteria for explicit error bounds and an exponential rate of convergence of the magic point empirical interpolation method introduced by \cite{BarraultNguyenMadayPatera2004}. Furthermore, we investigate its application to parametric integration.
 We find that the method is well-suited to Fourier transforms and has a wide range of applications in such diverse fields as probability and statistics, signal and image processing, physics, chemistry and mathematical finance. 
 To  illustrate the method, we apply it to the evaluation of probability densities by parametric Fourier inversion. Our numerical experiments display convergence of exponential order, even in cases where the theoretical results do not apply.

\end{abstract}




\section{Introduction}


At the basis of a large variety of mathematical applications lies the computation of parametric integrals of the form
\begin{equation}\label{parametric-integral}
\OI(h_p):=\int_{\Omega} h_p(z)\mu(\dd z)\qquad\text{for all }p\in\OP,
\end{equation}
where $\Omega\subset\ccd$ is a compact integration domain,\tild$\OP$ a compact parameter set, $(\Omega,\mathfrak{A}, \mu)$ a measure space with $\mu(\Omega)<\infty$ and $h:\OP\times\Omega\to\cc$ a bounded function.

A prominent class of such integrals are parametric \emph{Fourier transforms} 
\begin{equation}\label{para-Fourier-integral}
\widehat{f_q}(z) = \int_{\Omega} \ee{i\skl z,x \skr} f_q(x) \dd x\qquad\text{for all }p=(q,z)\in\OP
\end{equation}
with a parametric family of complex functions $f_q$ with compact support $\Omega$.
Today, Fourier transforms lie at the heart of applications in 
optics, electric engineering, chemistry, probability, partial differential equations, statistics and finance. 
The application of Fourier transforms even sparked many advancements in those disciplines, underlining its impressive power.
 The introduction of Fourier analysis to image formation theory in \cite{Duffieux1946} marked a turning point in optical image processing, as outlined for instance in \cite{Stark1982}.
The application of Fourier transform to nuclear magnetic resonance in \cite{ernst1966} was a major breakthrough in increasing the sensitivity of NMR, for which Richard R.~Ernst was awarded the Nobel prize in Chemistry in 1991. Fourier analysis also plays a fundamental role in statistics, in particular in the spectral representation of time series of stochastic data, see \cite{BrockwellDavis2002}. For an elucidation of other applications impacted by Fourier theory we recommend appendix~1 of \mbox{\cite{Kammler2007}}. 
 
For parametric Fourier integrals of form\tild\eqref{para-Fourier-integral} with fixed value of $q$ and $z$ being the only varying parameter, efficient numerical methods have been developed based on discrete Fourier transform (DFT) and fast Fourier transform (FFT).
The latter was developed by \cite{CooleyTukey1965} to obtain the Fourier transform $\widehat{f}(z)$ for a large set of values $z$ simultaneously. In regard to \eqref{para-Fourier-integral}, this is the special case of a parametric Fourier transform where $q$ is fixed and $z$ varies over a specific set. 
The immense impact of FFT highlights the exceptional usefulness of efficient methods for parametric Fourier integration. 

%

Shifting the focus to 
other examples, parametric integrals arise as generalized moments of parametric distributions in probability, statistics, engineering and computational finance. These expressions often appear in optimization routines, where they have to be computed for a large set of different parameter constellations. Prominent applications are model calibration and statistical learning algorithms based on moment estimation, regression and Expectation-Maximization (EM), see for instance \cite{HastieTibshiraniFriedman2009}.

The efficient computation of parametric integrals is also a cornerstone of the quantification of parameter uncertainty in generalized moments of form\tild \eqref{parametric-integral}, which leads to an integral of the form 
\begin{equation}\label{uncertain-parametric-integral}
\int_{\OP}  \int_{\Omega} h_p(z)\mu(\dd z) P(\dd p),
\end{equation}
where $P$ is a probability distribution on the parameter space $\OP$. 

In all of these cases, parametric integrals of the form \eqref{parametric-integral} have to be determined for a large set of parameter values. Therefore, efficient algorithms for parametric integration have a wide range of applicability. 
In the context of reduced basis methods for parametric partial differential equations, a magic point empirical interpolation method has been developed by \cite{BarraultNguyenMadayPatera2004} to treat nonlinearities that are expressed by parametric integrals.
The applicability of this method to the interpolation of parametric functions has been demonstrated by \cite{MadayNguyenPateraPau2009}. 

In this article, we focus on the approximation of parametric integrals by magic point empirical interpolation in general, a method that we call \emph{magic point integration} and
\begin{enumerate}
\item[---]
provide sufficient conditions for an \emph{exponential rate of convergence} of magic point interpolation and integration in its degrees of freedom, accompanied by \emph{explicit error bounds},
\item[---]
translate these conditions to the special case of parametric Fourier transforms, which shows the broad scope of the results, 
\item[---]
empirically prove efficiency of the method in a numerical case study.\footnote{An in-depth case study related to finance is presented in \cite{GassGlauMair2015} and supports our empirical findings.}
\end{enumerate}


 The remainder of the article is organized as follows. In section \ref{sec-magic} we revisit the magic point empirical interpolation method from \cite{BarraultNguyenMadayPatera2004} and present the related integral approximation, which can be perceived from two different perspectives. On the one hand it delivers a \emph{quadrature rule for parametric integrals} and, on the other, an \emph{interpolation method for parametric integrals} in the parameter space. In section \ref{sec-convergence} we provide analyticity conditions on the parametric integrals that imply exponential order of convergence of the method. We focus on the case of parametric Fourier transforms in the subsequent section\tild\ref{sec-magic-FT}. In a case study we discuss the numerical implementation and its results in section \ref{sec-case-study}. Here,
 we use the magic point integration method to evaluate densities of a parametric class of distributions that are defined through their Fourier transforms.

\section{Magic Point Empirical Interpolation for Integration}\label{sec-magic}
We now introduce the \emph{magic point empirical interpolation method for parametric integration} to approximate parametric integrals of the form \eqref{parametric-integral}.
Before we closely follow \citet{MadayNguyenPateraPau2009} to describe the interpolation method, let us state our basic assumptions that ensure the well-definedness of the iterative procedure. 

\begin{assumptions} \label{assumption:magic}
Let $\left(\Omega,\|.\|_{\infty}\right)$ and $\left(\OP,\|.\|_{\infty}\right)$ be compact, $\OP\times \Omega\ni (p,z)\mapsto h_p(z)$bounded and $p\mapsto h_p$ be sequentially continuous, i.e. for every sequence $p_i\to p$ we have  $\|h_{p_i}-h_{p}\|_\infty\to0$. Moreover, there exists $p\in\OP$ such that the function $h_p$ is not constantly zero. 
\end{assumptions}
For $M\in \nn$, the method delivers iteratively the \emph{magic points} 
$z^\ast_1,\ldots,z^\ast_M$, functions $\theta^M_1,\ldots,\theta^M_M$ and constructs the \emph{magic point interpolation} operator
\begin{align}
I_M(h)(p,z):= \sum_{m=1}^M h_{p}(z^\ast_m) \theta_m^M(z),
\end{align}
which allows us to define the \textbf{magic point integration} operator 
\begin{align}\label{Int_M}
\II_M(h)(p):= \sum_{m=1}^M h_{p}(z^\ast_m) \int_{\Omega} \theta_m^M(z) \dd z.
\end{align}

In the first step, $M=1$, we define the first \emph{magic parameter} $p^\ast_1$, the first \emph{magic point} $z^\ast_1$ and the first \emph{basis function} $q_1$ by
\begin{align}\label{def_p1}
p_1^\ast &:= \underset{p\in \OP}{\argmax}\|h_p\|_{L^\infty},
\\
 z^\ast_1 &:=\underset{z\in \Omega}{\argmax}|h_{p^\ast_1}(z)|
\label{def_z1},\\
 q_1(\cdot) &:=\frac{h_{p_1^\ast}(\cdot)}{h_{p_1^\ast}(z^\ast_1)}.\label{def_q1}
\end{align}
Notice that thanks to Assumptions \ref{assumption:magic}, these operations are well-defined. 

For $M\geq 1$ and $1\leq m\leq M$ we define
\begin{align}\label{def-theta}
\theta_m^M(z) := \sum_{j=1}^M(B^M)^{-1}_{jm} q_j(z), \qquad 
B^M_{jm}:=q_m(z^\ast_j),
\end{align}
where we denote by $(B^M)^{-1}_{jm}$ the entry in the $j$th line and $m$th column of the inverse of matrix $B^M$. By construction, $B^M$ is a lower triangular matrix with unity diagonal and is thus invertible. 

Then, recursively, as long as there are at least $M$ linearly independent functions in $\{h_p\,|\, p\in\OP\}$, the algorithm chooses the next magic parameter\tild $p_M^\ast$ according to a greedy procedure. It selects the parameter so that $h_{p_M^\ast}$ is the function in the set $\{h_p\,|\, p\in\OP\}$ which is worst represented by its approximation with the previously identified $M-1$ magic points and basis functions. So the $M$th magic parameter is 
\begin{align}\label{defu_M}
p^\ast_M := \underset{p\in \OP}{\argmax}\|h_p - I_{M-1}(h)(p,\cdot)\|_{L^\infty}.
\end{align}
In the same spirit, 
select the $M$th magic point as 
\begin{align}\label{defxi_M}
z^\ast_M:=\underset{z\in \Omega}{\argmax}\big|h_{p^\ast_M}(z) - I_{M-1}(h)(p^\ast_M,z)\big|.
\end{align}
The $M$th basis function $q_M$ is the residual $r_M$, normed to $1$ when evaluated at the new magic point, i.e.\
\begin{align}
r_M(z)&:= h_{p^\ast_M}(z) - I_{M-1}(h)(p^\ast_M,z),\\
 q_M(\cdot)&:=\frac{r_M(\cdot) }{r_M(z^\ast_M)}.\label{def-qM}
\end{align}
Notice the well-defined of the operations in the iterative step thanks to Assumptions\tild\ref{assumption:magic} and the fact that the denominator in \eqref{def-qM} is zero only if all functions in $\{h_p\,|\, p\in\OP\}$ are perfectly represented by the interpolation $I_{M-1}$, in which case they span a linear space of dimension $M-1$ or less.

\begin{remark}[Empirical quadrature rule for integrating parametric functions]\label{rem-snapshotprices}
\emph{
Magic point integration is an interpolation method for integrating a parametric family of integrands over a compact domain. From this point of view, the magic point empirical interpolation of \cite{BarraultNguyenMadayPatera2004} provides a \emph{quadrature rule for integrating parametric functions}. The weights are $\int_{\Omega} \theta_m^M(z) \dd z$ and the nodes are the magic points $z^\ast_m$ for $m=1,\ldots,M$. As in adaptive quadrature rules, these quadrature nodes are not fixed in advance. Whereas adaptive quadrature rules iteratively determine the next quadrature node in view of integrating a single function, the \emph{empirical integration method} \eqref{Int_M} tailors itself to the integration of a given \emph{parametrized family of integrands}.
}
\end{remark}

\begin{remark}[Interpolation of parametric integrals in the parameters]\label{rem-snapshotprices}
\emph{
For each $m=1,\ldots,M$ the function $\theta_m^M$ is a linear combination of snapshots $h_{p^\ast_j}$ for $j=1,\ldots,M$ with coefficients $\beta^m_j$, which may be iteratively computed. We thus may rewrite formula\tild \eqref{Int_M}\tild as
\begin{align}\label{magic-interpolation-integral}
\OI_M(h)(p) = \sum_{m=1}^M h_{p}(z^\ast_m) \sum_{j=1}^M \beta^m_j\int_\Omega h_{p_j^\ast}(z)\dd z.
\end{align}
Thus, for an arbitrary parameter $p\in\OP$, the integral $\int_\Omega h_{p}(z)\dd z$ is approximated by a linear combination of integrals $\int_\Omega h_{p_j^\ast}(z)\dd z$ for the magic parameters $p^\ast_j$. In other words, $\OI_M(h)$ is an interpolation of the function $p\mapsto \int_\Omega h_{p}(z)\dd z$. Here, the interpolation nodes are the magic parameters $p^\ast_j$ and the coefficients are given by $\sum_{m=1}^M h_{p}(z^\ast_m)\beta_j^m$. In contrast to classic interpolation methods, the "magic nodes" are tailored to the parametric set of integrands. As a striking advantage of this approach, the interpolation nodes are chosen independently of the dimensionality of the parameter space.
}
\end{remark}
%

A discrete approach to empirical interpolation has been introduced by \cite{ChaturantabutSorensen2010}. Whereas the empirical interpolation is designed for parametric functions, the input data for the Discrete Empirical Interpolation Method (DEIM) is discrete. A canonical way to use DEIM for integration is to first choose a fixed grid in $\Omega$ for a discrete integration of all parametric integrands. Then, DEIM can be performed on the integrands evaluated on this grid. 

In contrast, using magic point empirical interpolation for parametric integration separates the choice of the integration grid from the selection of nodes $p_m^\ast$. For each function, a different integration discretization may be used. Indeed, in our numerical study, we leave the discretization choices regarding integration to Matlab's \texttt{quadgk} routine.
Fixing a grid for discrete integration beforehand might for example become advantageous when the domain of integration $\Omega$ is high-dimensional.

\section{Convergence Analysis of magic point integration}\label{sec-convergence}

Theorem 2.4 in \cite{MadayNguyenPateraPau2009} relates the approximation error of magic point empirical interpolation to the best $n$-term approximation. We identify two generic cases for which this result implies exponential convergence of magic point interpolation and integration. In the first case, the set of parametric functions consists of univariate functions that are analytic on a specific Bernstein ellipse. In the second case, the functions may be multivariate and do not need to satisfy higher order regularity. The parameter set now is a compact subset of $\rr$ and the regularity of the parameter dependence allows an analytic extension to a specific Bernstein ellipse in the parameter space.

In order to formulate our analyticity assumptions, we define the 
\textit{Bernstein ellipse} $B([-1,1],\varrho)$ with parameter $\varrho>1$ as the open region in the complex plane bounded by the ellipse with foci $\pm 1$ and semiminor and semimajor axis lengths summing up to $\varrho$. Moreover, we define for $\ulb<\olb\in\rr$ the \textit{generalized Bernstein ellipse} by 
\begin{equation}
B([\ulb,\olb],\varrho):=\tau_{[\ulb,\olb]}\circ B([-1,1],\varrho),
\end{equation}
 where the transform $\tau_{[\ulb,\olb]}:\cc\to\cc$ is given by
	\begin{equation}
		\tau_{\underline{b},\overline{b}}(x) := \overline{b} + \frac{\underline{b}-\overline{b}}{2}(1-\Re(x))
 + i\,\frac{\overline{b}-\underline{b}}{2}\Im(x)\qquad \text{for all $x\in\IC$}.
	\end{equation}	 
For $\Omega \subset\rr$ we set
	\begin{equation}
		B(\Omega,\varrho) := B([\inf \Omega,\sup \Omega],\varrho). 
	\end{equation}

 \begin{defin}
Let $f:X_1\times X_2 \to \cc$ with $X_m\subset \cc^{n_m}$ for $m=1,2$. 
We say $f$ has the \emph{analytic property with respect to $\OX_1\subset \cc$ in the first argument}, if $X_1\subset \OX_1$ and there exist functions $f_1:X_1\times X_2 \to \cc$ and $f_2:X_2 \to \cc$ such that
\[
f(x_1,x_2) = f_1(x_1,x_2) f_2(x_2)\quad\text{for all }(x_1,x_2)\in X_1\times X_2
\]
and $f_1$
has an extension
$f_1: \OX_1\times X_2\to \cc$ such that for all fixed $x_2\in X_2$ the mapping $x_1\mapsto f(x_1,x_2)$ is analytic in the inner of $\OX_1$ and
\begin{equation}
	C_1 := \sup_{(x_1,x_2)\in \OX_1\times X_2}|f_1(x_1,x_2)|\sup_{x_2\in X_2} |f_2(x_2)|<\infty.
\end{equation}
We define the analytic property of $f$ with respect to $\OX_2$ in the second argument analogously.
 \end{defin}
We denote
 \begin{equation}
 \OI(h_p,\mu):= \int_{\Omega} h_p(z) \mu(\dd z)
 \end{equation}
 and for all $p\in\OP$ and $M\in\nn$,
\begin{equation}
 \II_M(h,\mu)(p):= \int_{\Omega} I_M(h)(p,z) \mu(\dd z),
 \end{equation}
 where $I_M$ is the magic point interpolation operator given by the iterative procedure from section\tild \ref{sec-magic}. Hence,\ 
 \begin{equation}
 \II_M(h,\mu)(p) = \sum_{m=1}^M h_p(z^\ast_m) \int_{\Omega} \theta^M_m(z) \mu(\dd z), 
 \end{equation}
 where $z^*_m$ are the magic points and $\theta^M_m$ are given by \eqref{def-theta} for every $m=1,\ldots,M$. 
 
 \begin{theorem}\label{theo-conv-magic-integration}
Let $\Omega\subset\cc^d$, $\OP$ and
$h:\OP\times\Omega\to\cc$ be such that Assumptions\tild\ref{assumption:magic} hold 
and one of the following conditions is satisfied for some\tild $\varrho>4$,
\begin{enumerate}
\item[(a)]
$d=1$, i.e. $\Omega$ is a compact subset of $\rr$, and~$h$ has the analytic property with respect to $B(\Omega,\varrho)$ in the second argument,
\item[(b)]
 $\OP$ is a compact subset of $\rr$, and~$h$ has the analytic property with respect to $B(\OP,\varrho)$ in  the first argument.
\end{enumerate} 
Then there exists a constant $C>0$ such that for all $p\in\OP$ and $M\in\nn$,
 \begin{align}
\big\|h-I_M(h)\big\|_{\infty} &\leq CM(\varrho/4)^{-M},\label{expo-h}\\
\sup_{p\in\OP}\big|\OI(h_p,\mu) - \II_M(h,\mu)(p)\big|&\leq C\mu(\Omega)M(\varrho/4)^{-M}.\label{expo-Int}
\end{align} 
\end{theorem}
\begin{proof}
Assume (a). Thanks to an affine transformation we may without loss of generality assume $\Omega=[-1,1]$. As assumed, there exist functions $f_1:\OP\times\Omega\to\cc$ and $f_2:\Omega\to\cc$ such that
\begin{equation}\label{eq-h-factors}
h_p(z) = f_1(p,z) f_2(z)\qquad \text{for all }p,z\in \OP\times\Omega
\end{equation}
and $f_1$ has an extension $f_1:\OP\times B(\Omega,\varrho)\to\cc$ that is bounded and for every fixed $p\in \OP$ is analytic in the Bernstein ellipse $B(\Omega,\varrho)$. We exploit the analyticity property to relate the approximation error to the best $n$-term approximation of the set $\OU:=\{h_p\,|\,p\in \OP\}$. This can conveniently be achieved by inserting an example of an interpolation method that is equipped with exact error bounds, and we choose Chebyshev projection 
for this task.
From Theorem\tild 8.2 in \cite{Trefethen2013} we obtain the explicit error bound, 
\begin{equation}\label{Cheby-error}
\sup_{p\in\OP}\big\| f_1(p,\cdot) - I^\text{Cheby}_N\big(f_1(p,\cdot)\big) \big\|_{\infty} \le c(f_1)M\varrho^{-N} ,
\end{equation} 
with constant $c(f_1):=\frac{2}{\varrho-1}\sup_{(p,z)\in \OP\times B(\Omega,\varrho)} \big|f_1(p,z)\big|$.

The Chebyshev projection of the family of functions $f_1(p,\cdot)$ induces an approximation of the family of functions $h_p$, along with an $N$-dimensional function space $\mathcal{U}_N$, simply by setting
\begin{equation}
I^K_N\big(h_p(\cdot) \big) (z) := I^\text{Cheby}_N\big(f_1(p,\cdot)\big)(z) f_2(z)
\end{equation} 
 for every $z\in\Omega$. From \eqref{Cheby-error}, the approximation $I^K_N$ inherits the error bound,
 \begin{equation}\label{IK-error}
\sup_{p\in\OP,z\in\Omega}\big| h_p(z) - I^K_N\big(h_p(\cdot)\big)(z) \big| \le c_1 \varrho^{-N} ,
\end{equation} 
with constant $c_1:=c(f_1)\max_{z\in\Omega}|f_2(z)|$. Using \eqref{IK-error}, we can apply the general convergence result from 
Theorem 2.4 in\tild\cite{MadayNguyenPateraPau2009}. Consulting their proof, we realize that
\begin{align*}
\sup_{p\in\OP}\big\|h_p-I_M(h)(p,\cdot)\big\|_{\infty} &\leq CM(\varrho/4)^{-M}
\end{align*} 
with $C=\frac{c_1\varrho}{4}$. Equation \eqref{expo-Int} follows by integration.

The proof follows analogously under assumption (b).
\end{proof}

\begin{remark}\label{rem-explixit-bound}
The proof makes the constant $C$ explicit. Under assumption (a) with representation\tild\eqref{eq-h-factors}, it is given by
\begin{equation}
C 
= \frac{\varrho}{2(\varrho-1)}\max_{(p,z)\in \OP\times B(\Omega,\varrho)} \big|f_1(p,z)\big| \max_{z\in\Omega}|f_2(z)|.
\end{equation}
Under assumption (b), an analogous constant can be derived.

\end{remark}

\begin{remark}\label{rem-MultiscaleIntegration} 
Consider a multivariate integral of the form
\begin{equation}
\OI_M(h,\mu_1\otimes\mu_2):=\int_{\OP}\int_{\Omega} h(p,z) \mu_1(\dd z) \mu_2(\dd p)
\end{equation}
with compact sets $\OP\subset\rr^D$ and $\Omega\subset\rrd$  equipped with finite Borel-measures $\mu_1$ and\tild$\mu_2$. Then, the application of the magic point empirical interpolation as presented in section\tild\ref{sec-magic} yields
\begin{equation}
 \II_M(h,\mu_1\otimes\mu_2) := \sum_{m=1}^M \int_{\OP} h_p(z^\ast_m) \mu_2(\dd p) \int_{\Omega} \theta^M_m(z)  \mu_1(\dd z),
 \end{equation}
and, under the assumptions of Theorem \ref{theo-conv-magic-integration}, we obtain
 \begin{align}
\big|\OI(h,\mu_1\otimes\mu_2) - \II_M(h,\mu_1\otimes\mu_2)\big|&\leq C\mu_1(\Omega)\mu_2(\OP)M(\varrho/4)^{-M}.\label{expo-IntInt}
\end{align} 
\end{remark}
If the domain $\Omega$ is one dimensional and $z\mapsto h_p(z)$, enjoys desirable analyticity properties, the \emph{approximation error} of magic point integration \emph{decays exponentially} in the number $M$ of interpolation nodes, \emph{independently of the dimension of the parameter space} $\OP$. Vice versa, if the parameter space is one dimensional and $p\mapsto h_p(z)$  enjoys desirable analyticity properties, the \emph{approximation error decays exponentially} in the number $M$ of interpolation nodes, \emph{independently of the dimension of the integration space}\tild$\Omega$. The reason why magic point interpolation and integration do not suffer from the curse of dimensionality -- in contrast to standard polynomial approximation -- is owed to the fact that classic interpolation methods are adapted to very large function classes, whereas magic point interpolation is adapted to a parametrized class of functions. 

In return, this highly appealing convergence comes with an additional cost: The offline phase involves a loop over optimizations in order to determine the magic parameters and the magic points. Whereas the online phase is independent of the dimensionality, the offline phase is thus affected. 
Let us further point out that Theorem\tild\ref{theo-conv-magic-integration} is a result on the \emph{theoretical iterative procedure} as described in section\tild\ref{sec-magic}. As we will discuss in section \ref{sec-complexity} below, the implementation inevitably involves additional problem simplifications and approximations, in order to perform the necessary optimizations. In particular, instead of the whole parameter space a training set is fixed in advance. For this more realistic setting, 
rigorous a posteriori error bounds have been developed for the empirical interpolation method, see \cite{EftangGreplPatera2010}. These bounds rely on derivatives in the parameters and straightforwardly translate to an error bound for magic point integration.

\section{Magic Points for Fourier Transforms}\label{sec-magic-FT}

In various fields of applied mathematics, Fourier transforms play a crucial role and parametric Fourier inverse transforms of the form
\begin{equation}\label{para-fourierint}
f_q(x) = \frac{1}{2\pi} \int_\rr \ee{-i z x} \widehat{f_q}(z)\dd z
\end{equation}
need to be computed for different values $p=(q,x)\in\OP=\OQ\times\OX$, where the function $z\mapsto\widehat{f_q}(z) := \int_\rr \ee{iz y} f_q(y)\dd y$ is well defined and integrable for all $q\in\OQ$.

A prominent example arises in the context of signal processing, where signals are filtered with the help of \emph{short-time Fourier transform} (STFT) that is of the form
\begin{equation}\label{STFT}
\Phi_b(f)(z) := \int_\rr f(t) w(t-b) \ee{-iz t} \dd t,
\end{equation}
with a \textit{window function} $w$ and parameter $b$. One typical example for windows are the Gau{\ss} windows, $w_\sigma(t) =\frac{1}{\sqrt{2\pi \sigma^2}}\ee{-t^2/(2\sigma^2)}$, where an additional parameter appears, namely the standard deviation  $\sigma$ of the normal distribution. The STFT with a Gau{\ss} window has been introduced by \cite{Gabor1946}. His pioneering approach has become indispensable for time-frequency signal analysis. We refer to \mbox{\cite{Debnath2014}} for historical backgrounds.


We consider the truncation of the integral in \eqref{para-fourierint} to a compact integration domain $\Omega=[\overline{\Omega},\underline{\Omega}]\subset\rr$ and choose the same domain $\Omega$ for all $p=(q,x)\in\OP=\OQ\times\OX$. Thus, we are left to approximate for all $p=(q,x)\in\OP$ the integral
 \begin{equation}\label{Paramteric-FT}
 \OI(h_p):= \int_{\Omega} h_p(z) \dd z,\, \text{where }h_p(z) =h_{(q,x)}(z)= \frac{1}{2\pi}  \ee{-i z x} \widehat{f_q}(z).
 \end{equation}
Then, according to section \ref{sec-magic}, we consider the approximation of $\OI(h_p)$ by magic point integration, i.e.\ by
\begin{equation}
 \II_M(h)(p) := \sum_{m=1}^M h_p(z^\ast_m) \int_{\Omega} \theta^M_m(z) \dd z, 
 \end{equation}
 where $z^*_m$ are the magic points and $\theta^M_m$ are given by \eqref{def-theta} for every $m=1,\ldots,M$. 
 In those cases where the analyticity properties of $q\mapsto  \widehat{f_q}(z)$ or of $z\mapsto  \widehat{f_q}(z)$ are directly accessible, Theorem \ref{theo-conv-magic-integration} can be applied to estimate the error $\sup_{(q,x)\in\OQ\times\OX}|\OI(h_{(q,x)})-\II_M(h)(q,x)|$. The following corollary offers a set of conditions in terms of existence of exponential moments of the functions $f_q$.
\begin{corollary} 
\label{cor:eins}
For $\eta>\sqrt{15/8}\,(\overline{\Omega}-\underline{\Omega})$ and every parameter $q\in\OQ$ assume $\int_\Omega \ee{(\eta+\varepsilon)|x|} \big|f_q(x) \big| \dd x <\infty$ for some $\varepsilon>0$ and assume further
\[
\sup_{q\in\OQ} \int_\Omega \ee{\eta|x|} \big|f_q(x) \big| \dd x <\infty.
\]
Then
 \begin{align}
\big\|h-I_M(h)\big\|_{\infty} &\leq C(\eta)M\big(\varrho(\eta)/4\big)^{-M},\label{expo-h-FT}\\
\sup_{(q,x)\in\OQ\times\OX}\big|\OI(h_{(q,x)}) - \II_M(h)(q,x)\big|&\leq C(\eta)|\Omega|M\big(\varrho(\eta)/4\big)^{-M},\label{expo-Int-FT}
\end{align}
with
\begin{align*}
\varrho(\eta) &= \frac{2\eta}{|\Omega|}+\sqrt{\left(\frac{2\eta}{|\Omega|}\right)^2+1},\\
 C(\eta) &= \frac{\varrho(\eta)}{4\pi(\varrho(\eta) - 1)} \ee{\eta( -\underline{\OX}\,\vee\,\overline{\OX})}\sup_{q\in\OQ} \int_\Omega \ee{\eta|x|} \big|f_q(x) \big| \dd x,
\end{align*}
where $|\Omega|:=\overline{\Omega}-\underline{\Omega}$, $\underline{\OX} = \inf \OX$ and $\overline{\OX} = \sup \OX$. 
\end{corollary} 

\begin{proof}
From the theorem of Fubini and the lemma of Morera we obtain that the mappings $z \mapsto \widehat{f_q}(z)$ have analytic extensions to the complex strip $\rr + i[-\eta,\eta]$.
We determine the value of $\varrho(\eta)$. It has to be chosen such that the associated semiminor $b_\varrho$ of the Bernstein ellipse $B([-1,1],\varrho)$ is mapped by $\tau_{[\underline{\Omega},\overline{\Omega}]}$ onto the semiminor of the generalized Bernstein ellipse $B([\underline{\Omega},\overline{\Omega}],\varrho)$ that is of length $\eta$. Using the relation
	\begin{equation}
		b_\varrho = \frac{\varrho - \frac{1}{\varrho}}{2}
	\end{equation}
between the semiminor $b_\varrho$ of a Bernstein ellipse and its ellipse parameter $\varrho$ we solve
	\begin{equation}
		\Im\big(\tau_{[\underline{\Omega},\overline{\Omega}]}(i\,b_\varrho)\big) = \frac{\overline{\Omega}-\underline{\Omega}}{2}\,b_\varrho = \eta
	\end{equation}
for $\varrho$ and find that for
	\begin{equation}
		\varrho(\eta)= \frac{2\eta}{\overline{\Omega}-\underline{\Omega}}+\sqrt{\left(\frac{2\eta}{\overline{\Omega}-\underline{\Omega}}\right)^2+1}
	\end{equation}
the Bernstein ellipse $B(\Omega,\varrho(\eta))$ is contained in the strip of analyticity of $\widehat{f_q}$ and by the choice of $\eta$ we have $\varrho(\eta)>4$. Hence assumption (a) of Theorem \ref{theo-conv-magic-integration} is satisfied. In regard to Remark \ref{rem-explixit-bound} we also obtain the explicit form of the constant $C(\eta)$, which proves the claim.
\end{proof}

Similarly, we consider the case where $\OQ$ is a singleton and
 \begin{equation}
 	h_p(z) = h_x(z) = \frac{1}{2\pi}\ee{-izx}\widehat{f}(z).
 \end{equation}
Under this assumption we obtain an interesting additional assertion if the integration domain $\Omega$ is rather small. This case occurs for example for approximations of STFT.
\begin{corollary} 
For $\eta>\sqrt{15/8}\,(\overline{\OX}-\underline{\OX})$ we have
 \begin{align}
\big\|h-I_M(h)\big\|_{\infty} &\leq C(\eta)M\big(\varrho(\eta)/4\big)^{-M},\\
\sup_{x\in\OX}\big|\OI(h_x)) - \II_M(h)(x)\big|&\leq C(\eta)|\Omega|M\big(\varrho(\eta)/4\big)^{-M},
\end{align}
with
\begin{align*}
\varrho(\eta) &= \frac{2\eta}{|\OX|}+\sqrt{\left(\frac{2\eta}{|\OX|}\right)^2+1},\\
 C(\eta) &= \frac{\varrho(\eta)}{4\pi(\varrho(\eta) - 1)} \ee{\eta( -\underline{\Omega}\, \vee\, \overline{\Omega})}\sup_{z\in\Omega} \big|\widehat{f}(z) \big|,
\end{align*}
where $|\OX|:=\overline{\OX}-\underline{\OX}$, $\underline{\OX} = \inf \OX$ and $\overline{\OX} = \sup \OX$. 
\end{corollary} 

The proof of the corollary follows similarly to the proof of Corollary\tild\ref{cor:eins}.

\section{Case Study}\label{sec-case-study}

\subsection{Implementation and complexity}
\label{sec-complexity}


The implementation of the algorithm requires further discretizations. For the parameter selection, we replace the continuous parameter space $\mathcal{P}$ by a discrete parameter cloud $\mathcal{P}^\text{disc}$. 
Additionally, for the selection of magic points we replace $\Omega$ by a discrete set $\Omega^\text{disc}$, instead of considering the whole continuous domain. Each function $h_p$ is then represented by its evaluation on this discrete $\Omega^\text{disc}$ and is thus represented by a finite-dimensional vector, numerically.

The crucial step during the offline phase is finding the solution to the optimization problem
\begin{align}\label{argmax}
\underset{(p,z)\in\OP^\text{disc}\times\Omega^\text{disc}}{\argmax}\bigg|h_{p}(z) - \sum_{m=1}^M h_{p}(z^\ast_m) \theta^M_m(z)\bigg|.
\end{align}

In a continuous setup, problem~\eqref{argmax} is solved by optimization routines. In our discrete implementation, however, we are able to consider \emph{all} magic parameter candidates in the discrete parameter space $\mathcal{P}^\text{disc}$ and \emph{all} magic point candidates in the discrete domain $\Omega^\text{disc}$ to solve problem~\eqref{argmax}. This results in a complexity of $\mathcal{O}(M\cdot |\OP^\text{disc}|\cdot |\Omega^\text{disc}|)$ during the offline phase for identifying the basis functions and magic points of the algorithm. We comment on our choices for the dimensionality of the involved discrete sets in the numerical section, later. 
Before magic point integration can be performed, the quantities $\int_\Omega \theta^M_m(z)\dd z$ need to be computed for $m=1,\ldots,M$. The complexity of this final step during the offline phase depends on the number of integration points that the integration method uses. 

\subsection{Tempered Stable Distribution}
\label{sec:TemperedStableEx}

We test the parametric integration approach on the evaluation of the density of a tempered stable distribution as an example. This is a parametric class of infinitely divisible probability distributions. A random variable $X$ is called infinitely divisible, if for each $n\in\nn$ there exist $n$ independent and identically distributed random variables whose sum coincides with $X$ in distribution. The rich class of infinitely divisible distributions, containing for instance the Normal and the Poisson distribution, is characterized through their Fourier transforms by the L\'evy-Khintchine representation. Using this link, a variety of infinitely divisible distributions is specified by Fourier transforms. The class of tempered stable distributions is an example of this sort, namely  it is defined by a parametric class of functions, which by the theorem of L\'evy-Khintchine are known to be Fourier transforms of infinitely divisible probability distributions. Its density function, however, is not  explicitly available. In order to evaluate it, a Fourier inverse has to be performed numerically. 
As we show below, magic point integration can be an adequate way to efficiently compute this Fourier inversion for different parameter values and at different points on the density's domain.

Following \cite{KuechlerTappe2014}, 
parameters
	\begin{equation}
		\alpha^+,\lambda^+,\alpha^-,\lambda^- \in (0,\infty),\qquad \beta^+,\beta^-\in (0,1)
	\end{equation}
define the distribution $\eta_q$ 
 on $(\IR,\OB(\IR))$ that we call a \emph{tempered stable distribution} and write
	\begin{equation}
		\eta_q = \text{TS}(q)=\text{TS}(\alpha^+,\beta^+,\lambda^+;\alpha^-,\beta^-,\lambda^-),
	\end{equation}
if its characteristic function $\varphi_q(z):=\int_{\rr}\ee{izx}\eta_q(\dd x)$ is given by
	\begin{equation}
	\label{eq:tscharfcn}
		\varphi_q(z) = \exp\left(\int_\IR (e^{izx}-1)F_q(\dd {x})\right),\quad z\in\IR,
	\end{equation}
with L{\'e}vy measure
	\begin{equation}
		F_q(\dd {x}) = \left(\frac{\alpha^+}{x^{1+\beta^+}}e^{-\lambda^+ x}\1_{x\in (0,\infty)} + \frac{\alpha^-}{x^{1+\beta^-}}e^{-\lambda^- x}\1_{x\in (-\infty,0)}\right) \dd x.
	\end{equation}
The tempered stable distribution is also defined for $\beta^+,\beta^-\in (1,2)$ where the characteristic function is given by a similar expression as in \eqref{eq:tscharfcn}. We consider the tempered stable distribution in a special case by introducing the parameter space
	\begin{equation}
	\label{eq:Qcgmy}
		\OQ = \{(C,G,M,Y)\ |\ C>0,\,G>0,\,M>0,\,Y\in(1,2)\}
	\end{equation}
and setting
	\begin{equation}
		\begin{split}
			\alpha^+ = \alpha^- = C,\qquad
			\lambda^- =  G,\qquad
			\lambda^+ = M,		\qquad	
			\beta^+ = \beta^- = Y.
		\end{split}
	\end{equation}
The resulting distribution $\gamma= \text{CGMY}(C,G,M,Y)$ is also known as CGMY distribution and is well established in finance,  see~\cite{CarrGemanMadanYor2002}. The density function of a tempered stable or a CGMY distribution, respectively, is not known in closed form. With $q=(C,G,M,Y)\in\OQ$ of \eqref{eq:Qcgmy} and $z\in\IR$, its Fourier transform, however, is explicitly given by
	\begin{equation}
	\label{eq:tscf}
		\varphi_q(z) = \exp\left(C\Gamma(-Y)\big((M-iz)^Y- M^Y + (G+iz)^Y - G^Y\big)\right),
	\end{equation}
wherein $\Gamma$ denotes the gamma function. By Fourier inversion, the density $f_q$ can be evaluated,
	\begin{equation}
	\label{eq:tsdensFinv}
		f_q(x) = \frac{1}{2\pi} \int_\IR e^{-i z x} \varphi_q(z) \dd {z}= \frac{1}{\pi} \int_0^\infty \Re\big(e^{-i z x} \varphi_q(z)\big) \dd {z}\qquad \text{for all $x\in\IR$}.
	\end{equation}

\subsection{Numerical Results}
We restrict the integration domain of \eqref{eq:tsdensFinv} to $\Omega=[0,75]$ and apply the parametric integration algorithm on a subset of
	\begin{equation}
		\OP = \OQ\times\IR,
	\end{equation}
specified in Table~\ref{tab:CGMYparameterints}.
\begin{table}[]
\centering
\begin{tabular}{@{}lccccc@{}}
\toprule
         & $C$     & $G$     & $M$     & $Y$         & $x$ \\ \midrule
interval & $[1,5]$ & $[1,8]$ & $[1,8]$ & $1.1$ & $[-1,1]$    \\ \bottomrule
\end{tabular}
\caption{Parameter intervals for the numerical study. The parameter $Y$ is kept constant.}
\label{tab:CGMYparameterints}
\end{table}
We draw $4.000$ random samples from $\OP$ respecting the intervals bounds of Table~\ref{tab:CGMYparameterints} and run the offline phase until an $L^\infty$ accuracy of $10^{-12}$ is reached.
\begin{figure}
\centering
\makebox[0pt]{\includegraphics[scale=0.95]{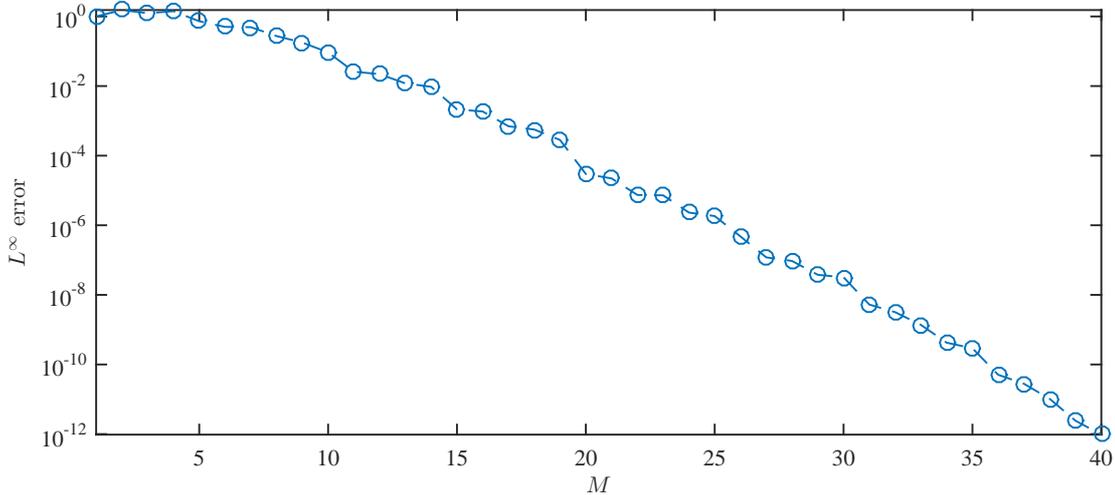}}
\caption{Decay of the error \eqref{defxi_M} on $\OP^\text{disc}$ and $\Omega^\text{disc}$ during the offline phase of the algorithm.}
\label{fig:tempStabOfflineErrordecay}
\end{figure}
The error decay during the offline phase is displayed by Figure~\ref{fig:tempStabOfflineErrordecay}. We observe exponential error decay reaching the prescribed accuracy threshold at $M= 40$.

Let us have a closer look at some of the basis functions $q_m$ that are generated during the offline phase. Figure~\ref{fig:tempStabMagicIntIntegrands} depicts five basis functions that have been created at an early stage of the offline phase (left) together with five basis functions that have been identified towards the end of it (right).
\begin{figure}
\centering
\makebox[0pt]{\includegraphics[scale=0.8]{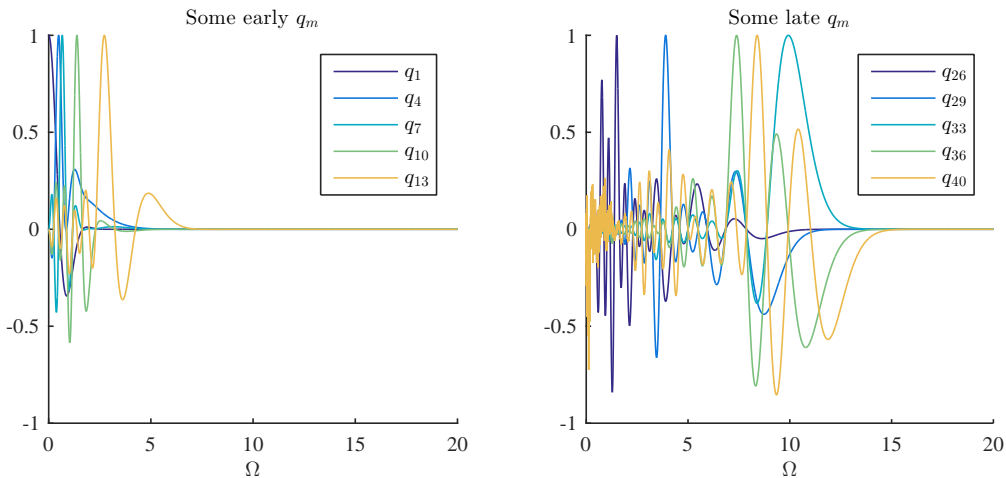}}
\caption{Some basis functions $q_m$ constructed early (left) and late (right) during the offline phase of the algorithm.}
\label{fig:tempStabMagicIntIntegrands}
\end{figure}
Note that all plotted basis functions are evaluated on a relatively small subinterval of $\Omega$. They are numerically zero outside of it. Interestingly, the functions in both sets have intersections with the $\Omega$ axis rather close to the origin. Such intersections mark the location of magic points. Areas on $\Omega$ where these intersections accumulate thus reveal regions where the approximation accuracy of the algorithm is low. In these regions, the shapes across all parametrized integrands seem to be most diverse.

In the right picture we observe in comparison that at a later stage in the offline phase, magic points $z^\ast_m$ are picked farer away from the origin, as well. We interpret this observation by assuming that the more magic points are chosen close to the origin, the better the algorithm is capable of approximating integrands more precisely there. Thus, its focus shifts towards the right for increasing values of $M$ where now smaller variations attract its attention.

We now test the algorithm on parameters not contained in the training set. To this extent we randomly draw $1.000$ parameter sets $p_i=(C_i,G_i,M_i,Y_i,x_i)$, $i=1,\dots,1.000$, uniformly distributed from the intervals given by Table~\ref{tab:CGMYparameterints}. For each $p_i$, $i=1,\dots,1.000$, we evaluate $f_{(C_i,G_i,M_i,Y_i)}(x_i)$ by Fourier inversion using Matlab's \texttt{quadgk} routine. Additionally, we approximate $f_{(C_i,G_i,M_i,Y_i)}(x_i)$ using the interpolation operator $\OI_{m}$ for all values $m=1,\dots,M$. For each such $m$ we compute the absolute and the relative error.

The $L^\infty((p_i)_{i=1,\dots,1.000})$ error for each $m=1,\dots,M$ is illustrated by Figure~\ref{fig:tempStabOOSerrdecay}.
\begin{figure}
\centering
\makebox[0pt]{\includegraphics[scale=1]{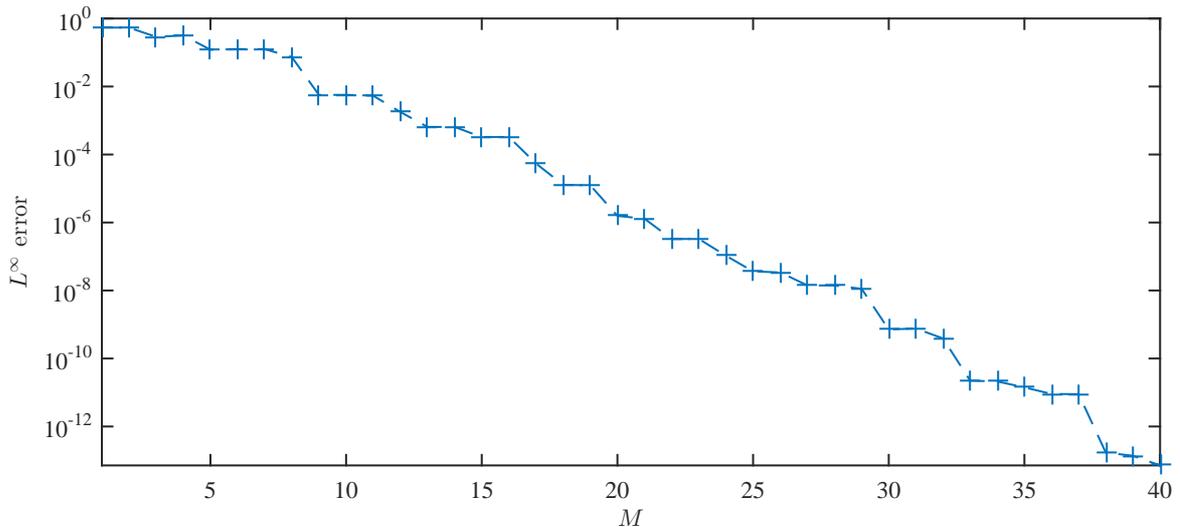}}
\caption{Out of sample error decay. For $1.000$ randomly drawn parameters $(C_i,G_i,M_i,Y_i,x_i)$, $i=1,\dots,1.000$ within the bounds prescribed by Table~\ref{tab:CGMYparameterints} the CGMY density is evaluated by the Parametric Integration method and compared to numerical Fourier inversion via Matlab's \texttt{quadgk} routine. the $L^\infty$ error decay is displayed.}
\label{fig:tempStabOOSerrdecay}
\end{figure}
\begin{figure}
\centering
\makebox[0pt]{\includegraphics[scale=1]{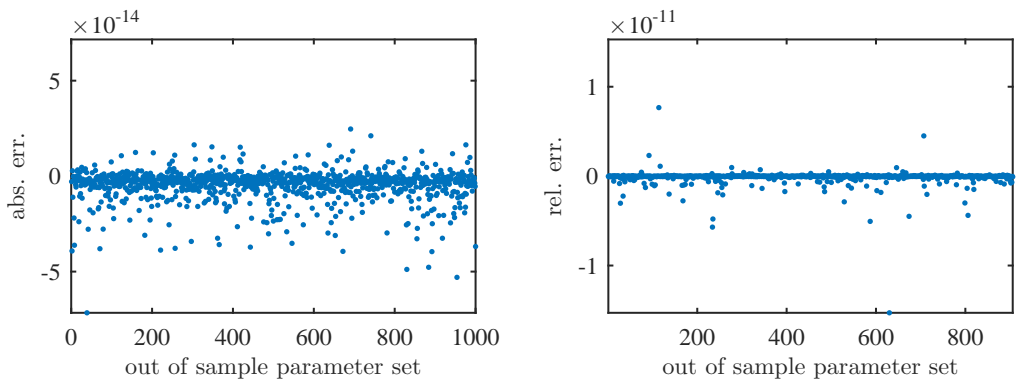}}
\caption{All out of sample errors in detail. Left: The absolute errors achieved for each of the $1.000$ randomly drawn parameter sets. Right: The relative errors for density values above $10^{-3}$.}
\label{fig:tempStabOOSallerrs}
\end{figure}
The setup of the numerical experiment does not fall in the scope of our theoretical result in several respects. We discretized both the parameter space and the set of magic point candidates. While Table\tild\ref{tab:CGMYparameterints} ensures a joint strip of analyticity $\IR+i(-1,1)$ that all integrands $h_p$ share, a Bernstein ellipse $B(\Omega, \varrho)$ with $\varrho>4$ on which those integrands are analytic does not exist. In spite of all these limitations, we still observe that the exponential error decay of the offline phase is maintained.\footnote{This observation is confirmed by our empirical studies in \cite{GassGlauMair2015}, where the existence of such a shared strip of analyticity was sufficient for empirical exponential convergence.} From this we draw two conclusions. First, both $\Omega^\text{disc}$ and $\OP^\text{disc}$ appear sufficiently rich to represent their continuous counterparts. Second, the practical use of the algorithm extends beyond the analyticity conditions imposed by Theorem\tild\ref{theo-conv-magic-integration}.

Figure~\ref{fig:tempStabOOSallerrs} presents the absolute errors and the relative errors for each parameter set $p_i$, individually, for the maximal value of $M$. Note that only relative errors for CGMY density values larger than $10^{-3}$ have been plotted to exclude the influence of numerical noise. While individual outliers are not present among the absolute deviations, they dominate the relative deviations in contrast. This result is in line with the objective function that the algorithm minimizes. 

Finally, Figure~\ref{fig:tempStabErrCloud} displays all magic parameters identified during the offline phase together with those randomly drawn parameter constellations that resulted in the ten smallest absolute errors (green dots) and the ten largest absolute errors (orange dots). We observe that orange parameter constellations are found in areas densely populated by magic parameters whereas green parameter constellations often do not have magic parameters in their immediate neighborhood, at all. This result may surprise at first. On second thought, however, we understand that areas where magic parameters accumulate mark precisely those locations in the parameter space where approximation is especially challenging for the algorithm. Integrands associated with parameter constellations from there exhibit the largest variation. Approximation for white areas 
on the other hand is already covered by the previously chosen magic parameters. Consequently, green dots are very likely to be found there.

\begin{figure}
\centering
\makebox[0pt]{\includegraphics[scale=1]{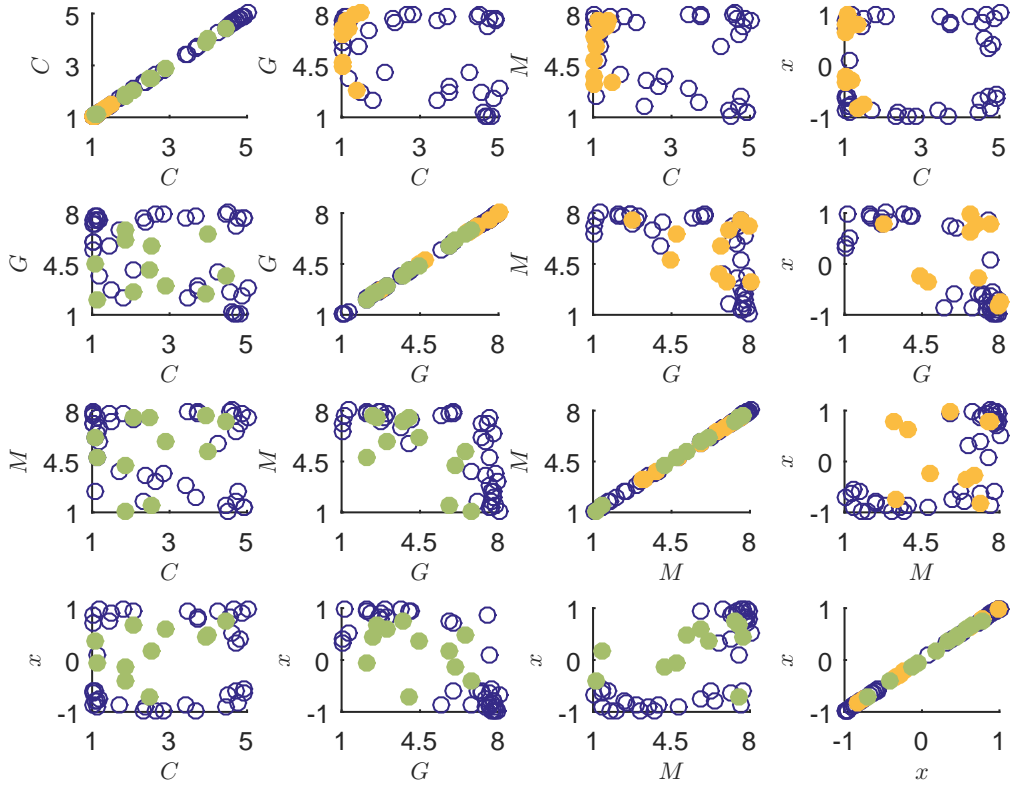}}
\caption{An illustration of those randomly drawn parameter constellations that resulted in the ten worst (orange) and the ten best (green) absolute errors. The empty blue circles are the magic parameters selected during the offline phase.}
\label{fig:tempStabErrCloud}
\end{figure}





\section*{Acknowledgement}
We thank Bernard Haasdonk, Laura Iapichino, Daniel Wirtz and Barbara Wohlmuth for fruitful discussions. Maximilian Ga{\ss} thanks the KPMG Center of Excellence in Risk Management  and Kathrin Glau acknowledges the TUM Junior Fellow Fund for financial support.

\bibliographystyle{elsarticle-harv}
  \bibliography{LiteraturFourierRB}

\begin{thebibliography}{17}
\expandafter\ifx\csname natexlab\endcsname\relax\def\natexlab#1{#1}\fi
\expandafter\ifx\csname url\endcsname\relax
  \def\url#1{\texttt{#1}}\fi
\expandafter\ifx\csname urlprefix\endcsname\relax\def\urlprefix{URL }\fi

\bibitem[{Barrault et~al.(2004)Barrault, Maday, Nguyen, and
  Patera}]{BarraultNguyenMadayPatera2004}
Barrault, M., Maday, Y., Nguyen, N.~C., Patera, A.~T., 2004. An ‘empirical
  interpolation’ method: application to efficient reduced-basis
  discretization of partial differential equations. Comptes Rendus
  Math{\'e}matique 339~(9), 667 -- 672.

\bibitem[{Brockwell and Davis(2002)}]{BrockwellDavis2002}
Brockwell, P.~J., Davis, R.~A., 2002. {Introduction to Time Series and
  Forecastung}, 2nd Edition. Springer.

\bibitem[{Carr et~al.(2002)Carr, Geman, Madan, and Yor}]{CarrGemanMadanYor2002}
Carr, P., Geman, H., Madan, D.~B., Yor, M., 2002. The fine structure of asset
  returns: An empirical investigation. The Journal of Business 75~(2),
  305--333.

\bibitem[{Chaturantabut and Sorensen(2010)}]{ChaturantabutSorensen2010}
Chaturantabut, S., Sorensen, D.~C., 2010. Nonlinear model reduction via
  discrete empirical interpolation. SIAM Journal on Scientific Computing
  32~(5), 2737--2764.

\bibitem[{Cooley and Tukey(1965)}]{CooleyTukey1965}
Cooley, J.~W., Tukey, J.~W., 1965. { An algorithm for the machine calculation
  of complex Fourier series}. Mathematics of Computation 19, 297--301.

\bibitem[{Debnath and Bhatta(2014)}]{Debnath2014}
Debnath, L., Bhatta, D., 2014. {Integral transforms and their applications},
  3rd Edition. {Apple Academic Press Inc.}

\bibitem[{Duffieux(1946)}]{Duffieux1946}
Duffieux, P.-M., 1946. { L'Int\'egrale de Fourier et ses Applications \`a
  l'Optique}. Oberthur.

\bibitem[{Eftang et~al.(2010)Eftang, Grepl, and Patera}]{EftangGreplPatera2010}
Eftang, J.~L., Grepl, M.~A., Patera, A.~T., 2010. A posteriori error bounds for
  the empirical interpolation method. {Comptes Rendus Math{\'e}matique}
  348~(9-10).

\bibitem[{Ernst and Anderson(1966)}]{ernst1966}
Ernst, R.~R., Anderson, W.~A., 1966. {Application of Fourier transform
  spectroscopy to magnetic resonance}. Review of Scientific Instruments 37~(1),
  93--102.

\bibitem[{Gabor(1946)}]{Gabor1946}
Gabor, D., 1946. {Theory of communication}. {Journal of the Institution of
  Electrical Engineers - Part III: Radio and Communication Engineering}
  93~(26), 429--457.

\bibitem[{Ga{\ss} et~al.(2015)Ga{\ss}, Glau, Mahlstedt, and
  Mair}]{GassGlauMair2015}
Ga{\ss}, M., Glau, K., Mahlstedt, M., Mair, M., 2015. Magic points in finance:
  empirical integration for parametrized option pricing, working paper.
\newline\urlprefix\url{arXiv:1511.00884}

\bibitem[{Hastie et~al.(2009)Hastie, Tibshirani, and
  Friedman}]{HastieTibshiraniFriedman2009}
Hastie, T., Tibshirani, R., Friedman, J., 2009. {The elements of Statistical
  Learning}, 2nd Edition. Springer.

\bibitem[{Kammler(2007)}]{Kammler2007}
Kammler, D.~W., 2007. {A first course in Fourier analysis}, 2nd Edition.
  Cambridge University Press.

\bibitem[{K\"uchler and Tappe(2014)}]{KuechlerTappe2014}
K\"uchler, U., Tappe, S., 2014. {Exponential stock models driven by tempered
  stable processes}. Journal of Econometrics 181~(1), 53--63.

\bibitem[{Maday et~al.(2009)Maday, Nguyen, Patera, and
  Pau}]{MadayNguyenPateraPau2009}
Maday, Y., Nguyen, C.~N., Patera, A.~T., Pau, G. S.~H., 2009. A general
  multipurpose interpolation procedure: the magic points. Communications on
  Pure and Applied Analysis 8~(1), 383--404.

\bibitem[{Stark(1982)}]{Stark1982}
Stark, H., 1982. {Theory and measurement of the optical Fourier transform}. In:
  Stark, H. (Ed.), Applications of Optical Fourier Transforms. Academic Press,
  pp. 2--40.

\bibitem[{Trefethen(2013)}]{Trefethen2013}
Trefethen, L.~N., 2013. Approximation Theory and Approximation Practice. SIAM
  books.

\end{thebibliography}
%
%

\end{document}